\documentclass[a4paper,20pt]{amsproc}

\usepackage{bbm}
\usepackage{hyperref}
\usepackage{color}
\newtheorem{theorem}{Theorem}[section]
\newtheorem{lemma}[theorem]{Lemma}
\newtheorem{corollary}[theorem]{Corollary}
\newtheorem{proposition}[theorem]{Proposition}

\theoremstyle{definition}
\newtheorem{definition}[theorem]{Definition}

\newtheorem{question}[theorem]{Question}

\theoremstyle{remark}

\usepackage{amssymb} 
\usepackage{graphicx} 

\newcommand{\R}{\mathbb{R}}

\numberwithin{equation}{section}

\begin{document}
\title{Extension of differentiable local mappings on linear topological spaces}


\author[Genrich Belitskii]{Genrich Belitskii}
\email{genrich@cs.bgu.ac.il}

\author[Victoria Rayskin]{Victoria Rayskin}
\email{vrayskin@gmail.com}

\providecommand{\keywords}[1]{\textbf{\textit{Index terms---}} #1}


\date{}
\maketitle
\begin{abstract}
Usually, for extension of local maps, one uses multiplication by so called bump functions. However, majority of infinite-dimensional linear topological spaces do not have
smooth bump functions. Therefore, in \cite{BR} we suggested a new approach for Banach spaces, based on the composition with locally identical maps. In the present work we discuss a possibility of generalization of this method for arbitrary spaces  and applications of this theory.
\end{abstract}

\section{Introduction}
Let $X$ and $Y$ be linear topological spaces. In this work, we will discuss the differentiable local maps $f: X\to Y$ and the possibility of differentiable extension of the maps. The map's extension is usually not unique and can be studied in the context of the equivalence class of $f$, i.e. a germ $[f]$.  Recall that a germ $[f]$ at $x\in X$  is the equivalence class of local maps, such that any pair of the class members coincides on some neighborhood of $x$. Each element of the class is called a representative of a germ. Occasionally, we denote germ $[f]$ as $f$. In the future, without loss of generality, we will assume that $x=0$.
We are interested in the question of existence of a global representative of the germ.
\\
Recall that $f: U\subset X \to Y$ is differentiable at  a point $x\in U$ if
$$
f(x+h) = f(x) + A\cdot h +r(h)
$$
with $r(h)= o(h)$ and with the continuous operator $A: X\to Y$.
\\
Each meaning of ``smallness'' $o(h)$ defines a corresponding specific notion of $r(h)$ and consequently a specific notion of differentiability. The value of the derivative is $A$ (usually denoted by $f'(x)$) and it is independent of the definition of $r(h)$.
Various definitions of differentiability are discussed in the works of 
F. and R. Nevanlinna (\cite{N}), 
H. R. Fischer (\cite{F}),
H. H. Keller (\cite{K}),
E. Вinz (\cite{B}), 
E. Binz, H. H. Keller (\cite{BK} ) and 
E. Binz, W. Meier - Solfrian (\cite{BM-S}). The main ideas of differentiation in abstract spaces were developed and clarified in the sequence of works by Averbuh, V. I. and Smoljanov, O. G (\cite{AS}), and later by M. Schechter in \cite{S}. We remind all these definitions and connections between them in Section~\ref{section-definitions}.
Thus, the question investigated in this work is the following.
\begin{question}
Let $f$ be a germ at $0$, differentiable in some a priori specified meaning. Does there exist a global differentiable (in the same sense) representative of the germ?
\end{question}

The classical method of a smooth extension of a locally defined map $f$ is multiplication by the smooth bump function. 
Recall that a smooth (differentiable) bump function is a real-valued function ranging between $0$ and $1$ with bounded support, which is equal to $1$ in a neighborhood $0$. 
Because it is known that very few spaces have smooth bump functions, we introduce\footnote{Earlier this idea was introduced in \cite{BR} for Banach spaces and for  differentiability in Fr\'echet sense. } here the new method based on the idea of considering the composition of $f$ with a locally identical map. Because we use composition in this construction, we will discuss only those notions of differentiability that satisfy the Chain Rule of differentiation.

In Section~\ref{section-main-prop} we show how locally defined maps with the help of our new method can be extended to the entire space. We present examples of the extension constructed for the specific Banach spaces (Section~\ref{section-banach}) and for the specific non-Banach linear topological spaces (Section~\ref{section-metric}). 
\\
In Section~\ref{section-applications} we discuss applications of this theory to the problem of conjugation with linear map.  We show that for the construction of differentiable conjugation, the assumption of the existence of smooth bump function is not necessary, and consequently the corresponding conjecture stated in the paper of W. Zhang, K. Lu and W. Zhang "Differentiability of the Conjugacy in the Hartman-Grobman Theorem" (\cite{ZLZ}) is incorrect.

\section{Background Definitions}\label{section-definitions}
Let us recall three definitions  of differentiation on linear topological spaces, which we use in this work. The reader can also find these definitions in~\cite{S}
\begin{itemize}
\item[1.] Bounded differentiability
\begin{definition}
 The map $f :X \to Y$ is bounded-differentiable at $x\in X$, if for every bounded subset $S\subset X$ and every $h\in S$ and $t\in \R$ 
$$r(th)/t \to 0$$ uniformly in $h$ as $t\to 0$.
\end{definition}

\item[2.] Compact (Hadamard) differentiability
\begin{definition}
The map $f:X \to Y$ is compact (Hadamard) differentiable at $x\in X$, if
     $$ f(x+t_n h_n) -f(x)=t_n A h +o(t_n)$$
as $t_n \to 0$, and $h_n \to h$.
\end{definition}

\item[3.] If both $X$ and $Y$ are Banach spaces with the norms $||.||_1$ and $||.||_2$ respectively, then Fr\'echet differentiation is well defined. 
\begin{definition}
The map $f$ is Fr\'echet differentiable at $0$ if
       $$ \lim_{h\to 0}||r(h)||_2/||h||_1 =0.$$
\end{definition}
\end{itemize}
Thus, as discussed in \cite{AS}, bounded differentiability implies the compact one.  The latter, in turn, is equivalent to Hadamard differentiability (see \cite{S}). If both $X$ and $Y$ are Banach then bounded and Fr\'echet differentiability coincide. For these types of differentiability many important rules, including the Chain Rule, hold.

The compact (Hadamard) differentiability is the weakest for which the Chain Rule is satisfied (for instance,  for the  G\^ateaux derivative the Chain Rule does not hold).



\section{The Main Proposition}\label{section-main-prop}

As explained above, we need to define differentiability, which satisfies the Chain Rule. Thus, we assume that the Cain Rule holds for the next definition and for the later discussion. 

In our work \cite{BR} we defined the notion of the blid map for Banach spaces, which stands for {\bf B}ounded {\bf L}ocal {\bf Id}entity map. In this paper we generalize this idea for linear topological spaces.

\begin{definition}
A space $X$ satisfies blid-differentiable property if for every neighborhood $U\subset X$ of $0$ there is a differentiable map $H$ defined on $X$, locally coinciding with the identity map, such that $H(X)\subset U$.
\end{definition}
Let us recall that a neighborhoods base of zero is a system $B=\{V_\alpha\}$ of neighborhoods of 0, such that for any neighborhood $U\subset X$ of 0 there exists some $V_\beta \in
 B$, $V_\beta \subset U$.

     Therefore, if there is a neighborhoods base $B$ such that for every $V_\alpha$ from $B$ there exists local identity $H_\alpha$, $H_\alpha(X) \subset  V_\alpha$, then $X$ satisfies the blid-property.
\begin{proposition}[The Main Proposition]\label{thm-main}
If $X$ satisfies blid-differentiable property, then every differentiable germ $[f]:X\to Y$ has a global differentiable representative.
\end{proposition}
\begin{proof}
 Let $f$ be a local representative of the germ defined on a neighborhood $U\subset X$ of zero. Let $H:X \to X$ be a differentiable local identity map such that

$H(X)\subset U$. Then the  map 
$$
                   F(x)=f(H(x)), \ \ x\in X
$$
is a global representative of the germ as we need. 
\end{proof}

\section{Banach Spaces}\label{section-banach}
In this section we will consider a Banach space $X$ and a general linear topological space $Y$. In  \cite{BR} we introduced the following definition:
\begin{definition}\label{def-blid map} A  differentiable {\it blid map} for a space $X$ is a global {\bf B}ounded {\bf L}ocal {\bf Id}entity differentiable  map $H:X \to X$.
\end{definition}

\begin{lemma}\label{lemma-blid-space}
If there exists differentiable blid map, then $X$ satisfies differentiable blid-property. 
\end{lemma}
\begin{proof}
It is convenient to chose the balls $B_c=\{x\in X: ||x||<c\}$ to be the base of neighborhoods in the Banach space $X$. Let $H$ be a blid map, such that $||H(x)||<N$. Then $$H_c(x)=\frac{c}{N} H(\frac{N}{c} x)$$ is the blid map as well and its image is inside of $B_c$. 
\end{proof}
The next statement immediately follows from the Lemma~\ref{lemma-blid-space} and the Main Proposition~\ref{thm-main}.
\begin{corollary}
 If the space $X$ admits differentiable blid map, then every differentiable germ at $0\in X$ into $Y$ has a global differentiable representative. 
\end{corollary}
Note, if $Y$ is a Banach space as well, then this result is proved  in \cite{BR} for differentiability in Fr\'echet sense.
\begin{corollary}
If the space $X$ has a differentiable bump function $h: X\to \R$, then it satisfies differentiable blid-property, and every differentiable germ at $0\in X$ has a global differentiable representative.
\end{corollary}
\begin{proof}
The map
$$
H(x) = h(x)x
$$
is a differentiable blid map for $X$.
\end{proof}

\subsection{The Spaces of Bounded Continuous Functions}

In this section we will consider the space $X=C(T)$ of bounded continuous functions on a topological space $T$ with the norm $||.|| = \sup_T |.|$. 

In the space $X$ we will construct a blid map in the following way.
$$
H(x)(t)=h(x(t))x(t)
$$
where $h$ is a smooth bump function on the real line. This blid map is  Fr\'echet- (consequently bounded-, consequently compact-) differentiable.
\begin{corollary}\label{thm-c-0-blid}
Any bounded- (consequently compact-) differentiable germ at $0\in C(T)$ has a global (in the corresponding sense) differentiable representative.
\end{corollary}
\subsection{The Spaces of Smooth Functions}
In this section we will consider $X=C^q[0,1]$, $0< q < \infty$, equipped with the norm $||x||_q= \max_{0\leq j\leq q}\sup_{t\in[0,1]}|x^{(j)}(t)|$. Recall that on the real line all types of differentiability coincide, and consequently can be viewed as Fr\'echet differentiability. We will see that $X$ has a differentiable blid map, and consequently satisfies differentiable blid property.
\begin{lemma}\label{lemma-C-q}
The space $C^q[0,1]$ admits  Fr\'echet-differentiable blid map.
\end{lemma}
\begin{proof}
The map 
$$
H(x)(t)=\sum_{j=0}^{q-1}\frac{t^j}{j!}h(x^{(j)}(0))x^{(j)}(0) + \int_0^t\,dt_1 \int_0^{t_1}\,dt_2... \int_0^{t_{q-1}} h\left( x^{(q)}(s) \right)x^{(q)}(s) \,ds ,
$$
where $h$ is a $C^{\infty}$ bump function on $\R$, represents a bounded differentiable blid map.
\end{proof}

\begin{corollary}
Any  bounded- (compact-) differentiable germ at $0\in C^q[0,1]$ has a global (in the corresponding sense) differentiable representative.
\end{corollary}

\section{Metric Spaces}\label{section-metric}
Let $X$ be a metric space with a metric $d(x,y)$.  Here we consider germs of maps from $X$ into an arbitrary linear  topological space $Y$.  Although instead of Fr\'echet differentiation (which is not defined for general metric spaces) we use bounded and compact (Hadamar) differentiation. The neighborhoods base $B$ can be chosen as a collection $\{B_c\}_c=\{x\in X: d(x,0)<c\}_c$. Then the space $X$ satisfies differentiable-blid property if for every $c$ there exists a differentiable, local identity map $H_c:X\to X$ such that $d(H_c(x),0)<c$ for all $x$, i.e., $H_c(X)\subset B_c$. 

In particular, if topology on $X$ is defined by countable collection of norms $||x||_{k}$, then the metric can be written as
$$
d(x,y):= \sum_{k=0}^\infty{\frac{1}{2^k}\cdot \frac{||x-y||_k}{||x-y||_k+1} }.
$$

\begin{lemma} Suppose for every $k=0,1,...$ there exists a global differentiable local identity map $H_k$ such that
$$
            \sup_x ||H_k(x)||_k< \infty.
$$
Then $X$ satisfies the differentiable blid property.
\end{lemma}
\begin{proof} For a given $c>0$ choose any 
\begin{equation}\label{k-c-inequality}
k>1- \ln c/\ln2
\end{equation}
and let $H_k$ be such that
$$
                      ||H_k(x)||_k <N,\ x\in X.
$$
Set 
$$
             H_c(x)=\frac{c}{4N}H_k\left(\frac{4N}{c} x\right).
$$
Then inequality~\ref{k-c-inequality} and the fact that $||x||_j$ is monotonically increasing with $j$ imply that   
$$
                d(H_c(x),0)<c,
$$
i.e. $H_c(X)\in B_c$.
\end{proof}

\subsection{The Space of Smooth Functions on the Real Line}
The space $X=C^q(\R)$ ($0\leq q <\infty$) of all smooth functions on $\R$ is endowed with the collection of norms
$$
               ||x||_k=\max_{t\in[-k,k]}\max_{l\leq q}|x^{(l)}(t)|.
$$
\begin{lemma}
The space $X$ possesses the bounded- (consequently compact-)  differentiable blid property.
\end{lemma}
\begin{proof}
Let $h(u)$ be a $C^{\infty}$-bump function on $\R$, which equals to 1 in a neighborhood of 0 and such that $a=\sup_{u\in\R}h(u)u<\infty$. Then 
$$
                H(x)(t)=\left\{
\begin{array}{l}
h(x(t))x(t),\  q=0\\
\sum_{j=0}^{q-1}\frac{t^j}{j!}h(x^{(j)}(0))x^{(j)}(0) + \int_0^t\,dt_1 \int_0^{t_1}\,dt_2 ... \int_0^{t_{q-1}} h_a\left( x^{(q)}(s) \right)x^{(q)}(s) \,ds,\  q\geq 1
\end{array}
\right.
$$
is differentiable local identity map, and
$$
                      ||H(x)||_k< ae^k,\  k=0,1,..., \ x\in X.
$$
\end{proof}
\begin{corollary} Every bounded- (consequently compact-) differentiable germ at $0\in C(\R)$ has a global  differentiable (in the corresponding sense) representative.
\end{corollary}
\subsection{The Space of Infinitely Differentiable Functions on a Closed Interval}

The space $X=C^{\infty}[0,1]$ is endowed with the collection of norms 
$$
 ||x||_k= \max_{j\leq k}\max_{t\in[0,1]}|x^{(j)}(t)|.
$$
\begin{lemma}
The space $X$ possesses the bounded- (consequently compact-) differentiable property.
\end{lemma}
\begin{proof}
Let $h(u)$ be the same bump function on $\R$ as above. Then
$$
        H_{0}(x)(t)=h(x(t))x(t)
$$
is differentiable local identity map, and
$$
           ||H_{0}(x)||_0<a.
$$

   Further, let $k>0$. Then 
$$
H_{k}(x)(t)=\sum_{j=0}^{k-1}\frac{t^j}{j!}h(x^{(j)}(0))x^{(j)}(0) + \int_0^t\,dt_1 \int_0^{t_1}\,dt_2 ... \int_0^{t_{k-1}} h\left( x^{(k)}(s) \right)x^{(k)}(s) \,ds .
$$
is differentiable local identity map, and
$$
                   ||H_{k}(x)||_k< ae^k, \ \ k=0,1,...,\ \ x\in X.
$$
\end{proof}
\begin{corollary}
Every bounded- (consequently compact-) differentiable germ at $0\in C^{\infty}[0,1]$ has a global representative. 
\end{corollary}

\subsection{The Space of Infinitely Differentiable Functions on the Real Line}

The space $X=C^{\infty}(\R)$ is endowed with the collection of norms 
$$
 ||x||_{k}= \max_{j\leq k}\max_{t\in[-k,k]}|x^{(j)}(t)|, \ k=0,1,2,...
$$
\begin{lemma}
The space $X$ possesses the bounded- (consequently compact-) differentiable property.
\end{lemma}
\begin{proof}
Let $h (u)$ be the same bump function on $\R$ as above. Then
$$
        H_{0}(x)(t)=h(x(t))x(t)
$$
is differentiable local identity map, and
$$
           ||H_{0}(x)||_0<a.
$$

   Further, let $k>0$. Then 
$$
H_{k}(x)(t)=\sum_{p=0}^{k-1}\frac{t^p}{j!}h(x^{(p)}(0))x^{(p)}(0) + \int_0^t\,dt_1 \int_0^{t_1}\,dt_2 ... \int_0^{t_{k-1}} h\left( x^{(k)}(s) \right)x^{(k)}(s) \,ds .
$$
is differentiable local identity map, and
$$
                  ||H_{k}(x)||_k< ae^k, \ \ k=0,1,...,\ \ x\in X.
$$
\end{proof}
\begin{corollary}
Every bounded- (consequently compact-) differentiable germ at $0\in C^{\infty}(\R)$ has a global representative. 
\end{corollary}

In conclusion, we would like to pose the following

\begin{question}\label{questionC-1} Which linear topological spaces have differentiable blid property? 
\end{question}
This question was not answered even for Banach spaces.

\section{Applications}\label{section-applications}
Local linearization and normal forms are convenient simplification of complex dynamics. In this section we discuss differentiable linearization on Banach spaces.
For a diffeomorphism $F$ with a fixed point $0$, we would like to find a smooth transformation $\Phi$ defined in a neighborhood of $0$ such that $\Phi\circ F \circ \Phi^{-1}$ has a simplified (polynomial) form called the normal form. If $\Phi\circ F \circ \Phi^{-1}=DF=\Lambda $ is linear, the conjugation is called linearization. There are two major questions in this area of research: how to increase smoothness of the conjugation $\Phi$, and whether it is sufficient to assume low smoothness of the diffeomorphism $F$.

Hartman and Grobman independently showed that if $\Lambda$ is hyperbolic, then for a diffeomorphism $F$ there exists a local homeomorphism $\Phi$ such that $\Phi \circ F \circ \Phi^{-1}=\Lambda$. Different proofs were given by Pugh \cite{P}. A higher regularity of $\Phi$ has been an active area of research.

The first attempt to answer the question of differentiability of $\Phi$ at the fixed point $0$ under hyperbolicity assumption was made in [48], but an error was found and discussed in \cite{R}. Later, in \cite{GHR}, Guysinsky, Hasselblatt and Rayskin presented correct proof. However, it was restricted to $F\in C^{\infty}$ (or more precisely, it was restricted to $F\in C^k$, where $k$ is defined by complicated expression). It was conjectured in the paper that the result is correct for $F\in C^2$, as it was announced in \cite{vS}.

Zhang, Lu and Zhang (\cite{ZLZ}, Theorem 7.1) showed that for a Banach space diffeomorphism $F$ with a hyperbolic fixed point and $\alpha$-H{\"o}lder $DF$, the local conjugating homeomorphism $\Phi$
is differentiable at the fixed point. Moreover, 
$$
\Phi(x) = x+O(||x||^{1+\beta}) \mbox{\  and \ } \Phi^{-1}(x) = x+O(||x||^{1+\beta})
$$
as $x\to 0$, for certain $\beta \in (0,\alpha]$. 

There are two additional assumptions in this theorem. The first one is the spectral band width inequality. The authors explain that this inequality is sharp if the spectrum has at most one connected component inside of the unit circle in $X$, and at most one connected component outside of the unit circle in $X$. For the precise formulation of the spectral band width condition we refer the reader to the paper \cite{ZLZ}.
It is important (and it is pointed out in \cite{ZLZ}) that this is not a non-resonance condition. The latter is required for generic linearization of higher smoothness. 

The second assumption is the assumption that the Banach space must possess smooth bump functions. 
It is conjectured in the paper that the second assumption is a necessary condition. 

In this section we explain that this conjecture is not correct (see Theorem~\ref{thm-diff}). The bump function condition can be replaced with the less restrictive blid map condition. 
Blid maps allow to reformulate Theorem 7.1 in the following way:
\begin{theorem}\label{thm-diff}
Let $X$ be a Banach space possessing a differentiable blid map with bounded derivative.
Suppose $F:X\to X$ is a diffeomorphism with a hyperbolic fixed point,  $DF$ is $\alpha$-H{\"o}lder, and the spectral band width condition is satisfied.
Then, there exists local linearizing homeomorphism $\Phi$ which 
is differentiable at the fixed point. Moreover, 
$$
\Phi(x) = x+O(||x||^{1+\beta}) \mbox{\  and \ } \Phi^{-1}(x) = x+O(||x||^{1+\beta})
$$
as $x\to 0$, for certain $\beta \in (0,\alpha]$. 
\end{theorem}
In particular, we have the following 
\begin{corollary}
Let $X=C^q[0,1]$.
Suppose $F:X\to X$ is a diffeomorphism with a hyperbolic fixed point,  $DF$ is $\alpha$-H{\"o}lder, and the spectral band width condition is satisfied.
Then, the local conjugating homeomorphism $\Phi$
is differentiable at the fixed point. Moreover, 
$$
\Phi(x) = x+O(||x||^{1+\beta}) \mbox{\  and \ } \Phi^{-1}(x) = x+O(||x||^{1+\beta})
$$
as $x\to 0$, for certain $\beta \in (0,\alpha]$. 
\end{corollary}

Below we justify Theorem~\ref{thm-diff}
\begin{proof}
Zhang, Lu and Zhang showed that for the conclusion of their Theorem 7.1 it is enough to satisfy the inequalities 1 and 2 (see  \ref{ineq} below),  which are called condition (7.6) in their paper.

In order to apply the blid maps instead of bump functions to the inequalities \ref{ineq}, it is sufficient to construct a bounded blid map, which has only first-order bounded derivative. I.e., let blid map $H(x): X\to X$ be as follows:
\begin{equation}
\begin{array}{l}
\mbox{1. \ }  H(x) = x \mbox{\ for\ } ||x||<1\\
\mbox{2. \ } H\in C^1 \mbox{\ and\ } ||H^{(j)}(x)||\leq c_j,  \ j=0,1 .
\end{array}
\end{equation} 

The condition (7.6) of \cite{ZLZ} is:
\begin{equation}\label{ineq}
\begin{array}{l}
\mbox{1. \ } \sup_{x\in X}||DF(x) -\Lambda|| \leq \delta_{\eta}\\
\mbox{2. \ } \sup_{x\in V\setminus O}\left\{||DF(x) - \Lambda|| / ||x||^{\alpha}\right\} = M < \infty 
\end{array}
\end{equation} 

Let $DF -\Lambda = f$. Define
$$
\tilde{f}(x):= f\left( \delta H(x/ \delta) \right)
$$

We will show that if $f$ satisfies (7.6), then so does $\tilde{f}$.
$$
\sup_{x\in X}|| D\tilde{f}(x)|| \leq \sup_{x\in X} || D f(x) || \cdot \sup_{x\in X} || DH(x) || \leq \delta_{\eta} \cdot c_1.
$$
Thus, the first inequality of (7.6) holds for $\tilde{f}$. For the second inequality we have the following estimate:
$$
\frac{|| D \tilde{f}(x)||}{||x||^{\alpha}} \leq \frac{|| Df \left(\delta H(x/\delta)\right) ||}{|| \delta H(x/\delta) ||^{\alpha}} \cdot \left( \frac{|| \delta H(x/\delta) ||}{||x||}  \right)^{\alpha}.
$$
The second multiple is bounded, because for small $x$ (say, $||x/\delta||<\epsilon$ for some $\epsilon>0$) we have
$$
\frac{|| \delta H(x/\delta) ||}{||x||} < c_1 + o(1),
$$
while for $||x/\delta|| \geq \epsilon$
$$
\frac{|| \delta H(x/\delta) ||}{||x||} < c_0/\epsilon.
$$
I.e., $\frac{|| \delta H(x/\delta) ||}{||x||}$ is less than some constant $m$.
Then, 
$$
\sup_{x\in V\setminus O}\frac{|| D \tilde{f}(x)||}{||x||^{\alpha}} \leq
\sup_{0<||x||<\delta c_0}\left\{||D f(x)|| / ||x||^{\alpha}\right\} \cdot \sup_{x\in X}|| D H(x)|| \cdot m^{\alpha}
$$
$$=\sup_{0<||x||<\delta c_0}\left\{||D f(x)|| / ||x||^{\alpha}\right\}c_1\cdot m^{\alpha}.
$$

This quantity is bounded by $M c_1 m^{\alpha}$ if $\delta$ is sufficiently small.
\end{proof}

Other applications in the area of local analysis on Banach spaces possessing blid maps can be found in \cite{BR}.

A generalization of Theorem~\ref{thm-diff} might be possible for the case of linear topological spaces (e.g., space of smooth functions), which posses differentiable blid property. 


\end{document}